\documentclass[11pt]{article}
\usepackage{a4wide}

\usepackage[a4paper, margin=2.3cm]{geometry}
\usepackage{amsmath,amssymb,amsthm}
\usepackage{color}
\usepackage{parskip}
\usepackage{hyperref}
\usepackage{parskip}
\usepackage{setspace}
\usepackage{graphicx}

\usepackage{setspace}

\renewcommand{\le}{\leqslant}
\renewcommand{\ge}{\geqslant}
 
\numberwithin{equation}{section}

\theoremstyle{plain}
\newtheorem{theorem}{Theorem}[section]
\newtheorem*{theorem*}{Theorem}
\newtheorem{lemma}[theorem]{Lemma}

\theoremstyle{remark}

\theoremstyle{definition}

\newtheorem{definition}[theorem]{Definition}

\newcommand{\ds}{\displaystyle}

\newcommand{\N}{{\mathbb N}}

\def\F{\mathcal{F}}

\def \F{{\mathbb F}}

\allowdisplaybreaks
\begin{document}

\title{On the $k$-resultant modulus set problem on varieties\\ over finite fields}
\author{Minh Quy Pham\thanks{Department of Mathematics, Da Lat University, Vietnam. Email: p.minhquydl@gmail.com}}
\date{}
\maketitle

\begin{abstract}
    Let $V\subset \mathbb{F}_q^d$ be a \textit{regular} variety, $k\ge 3$ is an integer and $A\subseteq V$. Covert, Koh, and Pi (2017) proved the following generalization of the Erd\H{o}s-Falconer distance problem: If $|A|\gg  q^{\frac{d-1}{2}+\frac{1}{k-1}}$, then  we have 
\[\Delta_{k}(A)=\{|x_1+\cdots+x_k|\colon x_i\in A\}\supseteq \mathbb{F}_q^*.\]
In this paper, we provide improvements and extensions of their result.
\end{abstract}

\section{Introduction}
Let $\mathbb{F}_q$ be a finite field of odd prime power $q$. Let $S_j^{d-1}$ be the sphere centered at the origin of radius $j$ in $\mathbb{F}_q^d$. For any two points $x=(x_1,\dots, x_d)$ and $y=(y_1,\dots, y_d)$ in $S_j^{d-1}$, the distance between them is defined by 
\begin{align*}
    | x-y| =(x_1-y_1)^2+\cdots+ (x_d-y_d)^2.
\end{align*}
For $A\subset S_j^{d-1}$,  we use $\Delta_2(A)$ to denote the set of all distances determined by any two points in $A.$ Hart, Iosevich, Koh and Rudnev \cite{HIKR}, using discrete Fourier analysis, proved the following result.
\begin{theorem}
Let $A$ be a subset of $S_1^{d-1}$.
\begin{itemize}
    \item[(i)] If $| A| \ge C_1q^{\frac{d}{2}}$ with a sufficiently large constant $C_1$, then there exists $C_2>0$ such that $| \Delta_2(A)| \ge C_2q.$
    \item[(ii)] If $d$ is even and $| A|\ge C_1q^{\frac{d}{2}}$ with a sufficiently large constant $C_1$, then $\Delta_2(A)=\mathbb{F}_q$.
    \item[(iii)] If $d$ is even, there exist $C_1>0$ and $A\subset S_1^{d-1}$ such that $| A| \ge q^{\frac{d}{2}}$ and $\Delta_2(A)\neq \mathbb{F}_q$.
    \item[(iv)] If $d$ is odd and $| A|\ge C_1q^{\frac{d+1}{2}}$ with a sufficiently large constant $C_1>0$, then $\Delta_2(A)=\mathbb{F}_q$.
    \item[(v)] If $d$ is odd, there exist $C_1>0$ and $A\subset S_1^{d-1}$ such that $| A| \ge cq^{\frac{d+1}{2}}$ and $\Delta_2(A)\neq \mathbb{F}_q$.
\end{itemize}
\end{theorem}
The key idea in their proof is to reduce the distance problem to the dot product problem by using the fact that $| x-y| =2-2x\cdot y$, where $x\cdot y=x_1y_1+\cdots +x_dy_d$. This is equivalent to say that
\begin{align}\label{eqconnection}
    | \Delta_2(A)| =| \Pi_2(A)| :=| \{ x\cdot y: x,y\in A\}|.
\end{align}
However, if one wants to replace the sphere by a general variety, the connection (\ref{eqconnection}) does not hold in general, so a new approach is needed.

In the paper \cite{Co}, Covert, Koh, and Pi were able to obtain an extension for \textit{regular} varieties.
\begin{definition}
For $A\subseteq \mathbb{F}_q^d$, we write $\textbf{1}_A$ for the characteristic function of $A$. Let $P(x)\in \mathbb{F}_q[x_1,\dots,x_d]$ be a polynomial. The variety $V:=\{ x\in \mathbb{F}_q^d: P(x)=0\}$ is called a regular variety if the size of $V$ is around $q^{d-1}$ and \begin{align*}
    \big| \widehat{\textbf{1}_V}(m)\big|=\bigg| \frac{1}{q^d}\sum\limits_{x\in \mathbb{F}_q^d}\chi(-m\cdot x)\textbf{1}_V(x)\bigg| \ll q^{-\frac{d+1}{2}}, \forall m\in \mathbb{F}_q^d\setminus \textbf{0}.
\end{align*}
\end{definition}

Throughout this paper, we use the following notations: $X\ll Y$ means that there exists some absolute constant $C_1>0$ such that $X\le C_1Y$ and $X\sim Y$ means $Y\ll X\ll Y$.

Covert, Koh, and Pi asked the following question: For $k\ge 2$, how large does a subset $A$ in a regular variety $V$ need to be to make sure that $\Delta_k(A)=\mathbb{F}_q$ or $| \Delta_k(A)| \gg q$, where $\Delta_k(A):=\{ x^1+\cdots+x^k: x^i\in A, 1\le i\le k\}$?

They observed that 
\begin{align*}
    | \Delta_k(A)| =| \Pi_k(A)|:=\bigg| \bigg\{ \sum\limits_{i=1}^k\sum\limits_{j=1}^k a_{ij}\cdot b_{ij}\cdot x^i\cdot x^j: x^l\in A, 1\le l\le k\bigg\}\bigg|,
\end{align*}
with $a_{ij}=1$ if $i<j$ and $0$ otherwise, and $b_{ij}=1$ for $i=1$ and $-1$ otherwise. Unlike the case of spheres, this reduction does not help with general varieties.

Using a new approach, Covert, Koh, and Pi proved the following theorem
\begin{theorem}[\cite{Co}]\label{thmcovert}
Let $V\subset \mathbb{F}_q^d$ be a regular variety, $k\ge 3$ is an integer and $A\subseteq V$. If $| A|\gg q^{\frac{d-1}{2}+\frac{1}{k-1}}$, then we have
\begin{align*}
    \Delta_k(A)\supseteq \mathbb{F}_q^*\text{ for even } d\ge 2,
\end{align*}
and 
\begin{align*}
    \Delta_k(A)=\mathbb{F}_q \text{ for odd } d\ge 3.
\end{align*}
\end{theorem}

This result gives us the important information that when $k$ is very large, in order to have $| \Delta_k(A)| \ge q-1$, one only needs the size of $A$ being around $q^{\frac{d-1}{2}}$. 

The main purpose of this paper to provide a variant of Theorem \ref{thmcovert} for a large family of varieties $V$ in which we only require some conditions on its dimensions and maximal affine subspaces in $V$, instead of the Fourier decay.

Our main results are as follows.

\begin{theorem}\label{thm1}
Let $d\ge 2$ and $V\subset \mathbb{F}_q^d$ be a variety which has dimension $n=\dim V\ge \frac{d+1}{2}$. Let $A$ be a subset of $V$ and $k\ge 3$ be an integer. There exists $\epsilon>0$ such that if $| A| \gg q^{\frac{d+1}{2}-\epsilon}$, then we have
    \begin{align*}
       | \Delta_k(A)| \gg q.
   \end{align*}
\end{theorem}

When $3\le k\le 4$, explicit exponents are obtained in the next theorem.

\begin{theorem}\label{thm2}
 Let $d\ge 2$ and $V\subset \mathbb{F}_q^d$ be a variety which has dimension $n=\dim V\ge \frac{d+1}{2}$. Set $t_V=\max \{ | H|: H \text{ is an affine subspace that lies in }V\}$.\\
 Suppose $t_V=q^\alpha$, $0\le \alpha\le \frac{d+1}{2}$, $\gamma=\big(2^{n+1}-n-5)^{-1}$ and let $A$ be a subset of $V$. 
 \begin{enumerate}
     \item[(i)]    If  $| A| \gg q^{\frac{d+1}{2}-\frac{\gamma((d+1)-2\alpha)}{2(2+\gamma)}}$, then
     \begin{align*}
         | \Delta_3(A)| \gg q.
     \end{align*}
     \item[(ii)] If $| A| \gg q^{\frac{d+1}{2}-\frac{\gamma((d+1)-2\alpha)}{2(1+\gamma)}}$, then
    \begin{align*}
        |\Delta_4(A)| \gg q.
    \end{align*}
 \end{enumerate}
\end{theorem}

When $V$ is a sphere with nonzero radius, we are able to obtain better results, which are also improvements of Theorem \ref{thmcovert}. 

\begin{theorem}[even dimensions]\label{thm3}
Suppose $d\ge 4$ be an even integer. Assume $j\neq 0$ and $A$ be a subset of $S_j^{d-1}$.
\begin{itemize}
    \item[(i)] For any $k>3$, there exists $C_k>0$ such that if $| A|\ge C_k q^{\frac{d-1}{2}+\frac{1}{4(k-2)}}$, then 
\begin{align*}
    | \Delta_k(A)| \gg q.
\end{align*}
    \item[(ii)] There exists $C_3>0$ such that if $| A|\ge C_3q^{\frac{d}{2}-\frac{1}{4}}$, then 
\begin{align*}
    | \Delta_3(A)| \gg q.
\end{align*}
\end{itemize}
\end{theorem}

Recall that an element $g\in \mathbb{F}_q$ is called a primitive element if $g$ is a generator of the group $\mathbb{F}_q^*$. We have the following theorem in odd dimensions.

\begin{theorem}[odd dimensions]\label{thm4}
Suppose $d\ge 3$ be an odd integer, $d=4l+1,l\in \N$ or $d=4l-1$, $q=1\mod 4$. Assume $j$ is a primitive element in $\F_q^\ast$, and $A$ be a subset of $ S_j^{d-1}$.
\begin{itemize}
    \item[(i)] For any $k>3$, there exists $C_k'>0$ such that if  $| A|\ge C_k' q^{\frac{d-1}{2}+\frac{1}{4(k-\frac{3}{2})}}$, then
\begin{align*}
    | \Delta_k(A)| \gg q.
\end{align*}
\item[(ii)] There exists $C_3'>0$ such that if $| A|\ge C_3' q^{\frac{d}{2}-\frac{1}{3}}$, then
\begin{align*}
    | \Delta_3(A)| \gg q.
\end{align*}
\end{itemize}
\end{theorem}
The rest of this paper is organized as follows: in Section 2, we will introduce some preliminary definitions and results which are needed for our proofs. Next,  Section 3 contains proofs of Theorems \ref{thm1} and \ref{thm2}. Theorems \ref{thm3} and \ref{thm4} are proved in Section 4.

\section{Additive energies and key lemmas}
The main purpose of this section is to introduce definitions of energies of sets and list related estimates which will be important in proofs of our main theorems.

Let us first recall the notion of the additive energy of sets in $\mathbb{F}_q^d$.
\begin{definition}
Let $A, B\subset \mathbb{F}_q^d$, the additive energy of $A$ and $B$, denoted $E(A,B)$, is defined by \begin{align*}
    E(A,B)=| \{(a,a',b,b')\in A^2\times B^2: a+b=a'+b'\}|.
\end{align*}
If $A=B$, we denote  the additive energy of $A$ by $E(A)$.
\end{definition}
\begin{definition}
Let $A\subset \mathbb{F}_q^d$ and  $k\ge 1$ be an integer. The $k-$energy of $A$, denoted $E_k(A)$, is defined by 
\begin{align*}
    E_k(A)=| \{ (x^1,\dots,x^{2k})\in A^{2k} : x^1+\dots+x^k=x^{k+1}+\dots+x^{2k}\}|.
\end{align*}
\end{definition}
Note that $E_1(A)=| A|$ and $E_2(A)=E(A)$.

In order to obtain our results, we would need to have good estimates for the energies for which we will use the following results. 
\begin{lemma}[Lemma 3.1, \cite{IKLPS}]\label{energydeven}
Suppose $d\ge 4$ is even, and $j$ is a nonzero element in $\mathbb{F}_q$. Let $A$ be a subset of $S_j^{d-1}$, then we have
\begin{align*}
    E(A)\ll \frac{| A|^3}{q}+q^{\frac{d-2}{2}}| A|^2.
\end{align*}
\end{lemma}

\begin{lemma}[Theorem 2.5, \cite{KPV}]\label{energydood}
Suppose either $d=4l-1$ and $q=1\mod 4$ or $d=4l+1, l\in \N$. Assume that $j$ is a primitive element in $\mathbb{F}_q^*$. Let $A$ be a subset of $S_j^{d-1}$, then we have
\begin{align*}
    E(A)\ll  \frac{| A|^3}{q}+q^{\frac{d-2}{2}}| A|^2.
\end{align*}
\end{lemma}

\begin{lemma}[Lemma 4.1, \cite{Co}]\label{energykinduction}
Let $A$ be a subset of $S_j^{d-1}$ and $k\ge 2$ be an integer, then we have
    \begin{align*}
        E_{k}(A)\ll q^{d-1}E_{k-1}(A) +\frac{| A|^{2k-1}}{q}.
    \end{align*}
\end{lemma}

Let $A, B\subset \mathbb{F}_q^d$, the distinct distances set between $A$ and $B$ is defined by
\begin{align*}
    \Delta_2(A,B):=\{ | x+y|: x\in A,y\in B\}.
\end{align*}
As we shall see in the proofs of our results, given $A\subset \mathbb{F}_q^d$, we shall relate $| \Delta_k(A)|$ to energies $E_l(A)$ and $| \Delta_2(A,B)|$, where $B$ is an arbitrary set in $\mathbb{F}_q^d$. Therefore, we also recall some known results regarding distances of two sets in $\mathbb{F}_q^d$.

\begin{lemma}[Theorem 3.5 (1), \cite{KSu}]\label{devenAB2}
Let $d\ge 2$ is even. Let $A,B\subset \mathbb{F}_q^d$ such that $| A| | B| \ge 16q^d$ and $q^{\frac{d-1}{2}}\le | A| < q^{\frac{d+1}{2}}$. 
If $| B| \ge 2q^{\frac{d+1}{2}}$ then we have
\begin{align*}
    | \Delta_2(A,B)|\gg q.
\end{align*}
\end{lemma}

\begin{lemma}[Theorem 1.19, \cite{IKLPS}]\label{doodAB}
Let $j$ be a non-square number of $\mathbb{F}_q^\ast$. Assume $d=4l+1, l\in \N$, or $d=4l-1$ and $q\equiv 1\mod 4$. Let $A\subset S_j^{d-1}$ and $B\subset \mathbb{F}_q^d$. If $| A| | B| \ge 4q^d$, then we have
\begin{align*}
    | \Delta_2(A,B)| \ge \frac{q}{4}.
\end{align*}
\end{lemma}

\begin{lemma}[Theorem 2.1, \cite{Shp}]\label{twosets}
Let $d\ge 1$, and let $A,B\subset \mathbb{F}_q^d$ such that $| A| | B| \gg q^{\frac{d+1}{2}}$. Then we have $| \Delta_2(A,B)| \gg q.$
\end{lemma}
Now for each positive integer $l\ge 1$ we define $A_l=\{ x^1+x^2+\dots+x^l: x^i\in A, 1\le i\le l\}$.

Let $k\ge 3$ be any integer, one can see that the $k-$distances set of $A$ can be represented as
\begin{align*}
    \Delta_k(A)=\Delta_2(A_l,A_{k-l}),\text{ for any } 1\le l\le k-1.
\end{align*}
Moreover, the lower bound for cardinality of $A_l$ can be determined via the $|A|$ and $E_{l}(A)$ as in the following lemma. This allows us to reduce our problem of determining $|A|$ such that $k-$distances set $\Delta_k(A)\gg q$ to the problem of finding conditions of $A$ and $E_{l}(A)$ such that $\Delta_2(A_l,A_{k-l})\gg q$. 

\begin{lemma}\label{cardak}
Let $d\ge 2$ and $A$ be a subset of  $\mathbb{F}_q^d$. Then for any integers $l\ge 1$, we have
\begin{align*}
    | A_l| \ge\frac{| A|^{2l}}{E_{l}(A)}.
\end{align*}
\end{lemma}
\begin{proof}
For each $y\in \mathbb{F}_q^d$, denote $\mu_l(y)=| \{ (x^1,\dots, x^l)\in A^{l}: x^1+\dots +x^l=y\}|$. Note that $\mu_l(y)>0$ if and only if $y\in A_l$.
By the Cauchy-Schwarz inequality, we have
\begin{align*}
    | A|^{l}
    &=\sum\limits_{y\in A_l} \mu_l(y)
    \le \bigg(\sum\limits_{y\in A_l}1\bigg)^{\frac{1}{2}}\bigg(\sum\limits_{y\in A_l} \mu_l^2(y)\bigg)^{\frac{1}{2}}\\
    &=| A_l|^{\frac{1}{2}}
    \bigg[\sum\limits_{y\in A_l} 
    \bigg(\sum\limits_{\substack{x^1+\dots +x^l=y\\ x^i\in A}}1\bigg)
    \bigg(\sum\limits_{\substack{z^1+\dots +z^l=y\\ z^i\in A}}1\bigg)
    \bigg]^{\frac{1}{2}}
    =| A_l|^{\frac{1}{2}}
    \bigg(\sum\limits_{y\in A_l} 
    \sum\limits_{\substack{x^1+\dots +x^l=y=z^1+\dots +z^l\\ x^i,z^i\in A}}1\bigg)^{\frac{1}{2}}\\
    &=| A_l|^{\frac{1}{2}}
    \bigg(
    \sum\limits_{\substack{x^1+\dots +x^l=z^1+\dots +z^l\\ x^i,z^i\in A}}1\bigg)^{\frac{1}{2}}
    =| A_l|^{\frac{1}{2}} E_{l}^{\frac{1}{2}}(A),
\end{align*}
where the last line holds by definition of $l-$energy of $A$. The lemma follows.
\end{proof}

\section{Proofs of Theorems \ref{thm1} and \ref{thm2}}
To prove Theorems \ref{thm1} and \ref{thm2}, we make use of the following result due to Shkredov in \cite{Sh21}.

\begin{theorem}[Corollary 9, \cite{Sh21}]\label{thmenvarsh}
Let $V$ be a variety of $\mathbb{F}_q^d$, and $n=\dim(V)$, $D=\deg (V)$. Let $A$ be a subset of $\mathbb{F}_q^d$ such that $A\subset V$.
\begin{itemize}
    \item[(i)] If $n\ge 2$, one has
\begin{align*}
    E(A)\ll_{n,D}| A|^3\bigg(\frac{t}{| A| }\bigg)^{(2^{n+1}-n-5)^{-1}}.
\end{align*}
If $n=1$, one has $E(A)\ll_{n,D} | A|^2t$.
\item[(ii)] For any $k\ge 1$ one has either
\begin{align*}
    E_k(A)
    \le | A|^{2k-1-\frac{c\beta}{4}}
    \hspace{1cm}\text{or }\hspace{1cm}
    E_{k+1}(A)
    \ll_{n,D} | A|^{2-\frac{c\beta}{4}} E_k(A),
\end{align*}
where $\beta=\beta(n)\in [4^{-n},2^{-n+1}]$ and $0<c\le 1$.
\end{itemize}
\end{theorem}

\subsection{Proof of Theorem \ref{thm1}}
In order to prove Theorem \ref{thm1}, we assume $d\ge 2$, $V\subset \mathbb{F}_q^d$ be an algebraic variety and denote $n=\dim V\ge\frac{d+1}{2}$. 

From the part $(ii)$ of Theorem \ref{thmenvarsh}, we know that for each $l\ge 2$, one of the followings holds
\begin{itemize}
    \item[(a)] $E_l(A)\ll | A|^{2l-1-c\beta/4}$ 
    \item[(b)] $E_{l+1}(A)\ll | A|^{2-c\beta/4}E_l(A)$,
\end{itemize}
where $\beta=\beta(n)\in [4^{-n}, 2^{-n+1}]$. We also know from \cite{Sh21} that there is $c\in (0,1]$ such that
\begin{align}\label{eqthm1.1}
    E_l(A)\ll | A|^{-\frac{c\beta}{2} +\frac{1}{2}+l-\frac{3\beta n}{2}+\beta l n}E_{l-1}^{\frac{1}{2}-\frac{\beta n}{2}}(A).
\end{align}
For $2\le l\le k$, denote
\begin{align*}
    \eta(l)=\max\{ j: E_{j} \text{ satisfies } (a), 2\le j\le l\}.
\end{align*}
We now consider three cases depending on the value of $\eta(l)$:\\
\textbf{Case 1.} If $\eta(l)=l$, then we have $E_l(A)\ll | A|^{2l-1-c\beta/4}$.

\textbf{Case 2.} If $\eta(l)=l-1$, then $E_{l-1}(A)\ll | A|^{2l-3-c\beta/4}$. By (\ref{eqthm1.1}), we have
\begin{align*}
    E_l(A)
    &\ll | A|^{-\frac{c\beta}{2} +\frac{1}{2}+l-\frac{3\beta n}{2}+\beta l n}E_{l-1}^{\frac{1}{2}-\frac{\beta n}{2}}(A)
    \ll | A|^{-\frac{c\beta}{2} +\frac{1}{2}+l-\frac{3\beta n}{2}+\beta l n}
    \big(| A|^{2l-3-c\beta/4}\big)^{\frac{1}{2}-\frac{\beta n}{2}}\\
    &=| A|^{2l-1-\frac{5c\beta}{8}+\frac{cn\beta^2}{8}}.
\end{align*}
\textbf{Case 3.} If $\eta(l)\le l-2$, by using $(b)$ and an inductive argument, we get
\begin{align}\label{eqthm1.2}
    E_{l}(A)\ll \big(| A|^{2-c\beta/4}\big)^{l-\eta(l)-1} E_{\eta(l)+1}(A).
\end{align}
On the other hand, it follows from (\ref{eqthm1.1}) that
\begin{align*}
    E_{\eta(l)+1}(A)
    &\ll | A|^{-\frac{c\beta}{2}+\frac{3}{2}+\eta(l)-\frac{\beta n}{2}+\beta \eta(l) n} E_{\eta(l)}^{\frac{1}{2}-\frac{\beta n}{2}}(A)\\
    &\ll | A|^{-\frac{c\beta}{2}+\frac{3}{2}+\eta(l)-\frac{\beta n}{2}+\beta \eta(l) n} \big( | A|^{2\eta(l)-1-c\beta/4}  \big)^{\frac{1}{2}-\frac{\beta n}{2}}\\
    &=| A|^{2\eta(l)+1-\frac{5c\beta}{8}+\frac{cn\beta^2}{8}}.
\end{align*}
Substituting this into (\ref{eqthm1.2}), one has
\begin{align*}
    E_l(A) 
    &\ll \big(| A|^{2-c\beta/4}\big)^{l-\eta(l)-1} | A|^{2\eta(l)+1-\frac{5c\beta}{8}+\frac{cn\beta^2}{8}}\\
    &= | A|^{2l-1+\frac{c\eta(l)\beta}{4}-\frac{c\beta l}{4}-\frac{3c\beta}{8}+\frac{cn\beta^2}{8}}
\end{align*}
Putting these three cases above together and using the fact that
\begin{align*}
    \min\bigg\{\frac{c\beta}{4},\frac{c\beta}{4}\bigg(\frac{5}{2}-\frac{n\beta}{2}\bigg),\frac{c\beta}{4}\bigg(l-\eta(l)+\frac{1}{2}(3-n\beta)\bigg)\bigg\}=\frac{c\beta}{4},
\end{align*}
one has
\begin{align*}
    E_l(A)\ll | A|^{2l-1-\frac{c\beta}{4}}, \forall l\ge 2.
\end{align*}
Hence, it follows from Lemma \ref{cardak} that
\begin{align*}
     | A_l|  \ge \frac{| A|^{2l}}{E_l(A)}\ge \frac{| A|^{2l}}{| A|^{2l-1-\frac{c\beta}{4}}}=| A|^{1+\frac{c\beta}{4}},
\end{align*}
which implies
\begin{align*}
    | A_l| | A_{k-l}| 
    \ge | A|^{2+\frac{c\beta}{2}}.
\end{align*}
To complete the proof, we observe that $\Delta_k(A)=\Delta_2(A_l,A_{k-l})$. So, it is sufficient to show that $| \Delta_2(A_l,A_{k-l})|\gg q$ under the condition $| A| \gg q^{\frac{d+1}{2}-\epsilon}$ for some $\epsilon>0$.\\
Indeed, choose $\epsilon=\frac{(d+1)c\beta}{2(4+c\beta)}>0$, and assume that $| A|\gg q^{\frac{d+1}{2}-\epsilon}=q^{\frac{d+1}{2+\frac{c\beta}{2}}}$, then we can check that
\begin{align*}
    | A_l| | A_{k-l}| 
    \ge q^{\frac{d+1}{2+\frac{c\beta}{2}}(2+\frac{c\beta}{2})}\gg q^{d+1}.
\end{align*}
Applying Theorem $\ref{twosets}$ for two sets $A_l$ and $A_{k-l}$, one obtains $| \Delta_k(A)| =| \Delta_2(A_l,A_{k-l})| \gg q$.
\begin{flushright}
$\Box$
\end{flushright}

\subsection{Proof of Theorem \ref{thm2}}
Let $d\ge 2$ and $V\subset \mathbb{F}_q^d$ be a variety with $n=\dim V\ge \frac{d+1}{2}$. Denote $t_V=q^\alpha$, where $0\le \alpha\le \frac{d+1}{2}$. Also, set $\gamma=(2^{n+1}-n-5)^{-1}$. 
 
Let $A$ be a subset of $V$.
\begin{enumerate}
    \item[(i)] Assume $k=3$ and $| A|\gg q^{\frac{d+1}{2}-\frac{\gamma((d+1)-2\alpha)}{2(2+\gamma)}}$. \\
    Set $A_1=A$ and $A_2=A+A$. Applying Lemma \ref{cardak} and the part $(i)$ of Theorem \ref{thmenvarsh}, we get
    \begin{align*}
        | A_1| | A_2|
        &\ge | A|\frac{| A|^4}{E_2(A)}
        \gg \frac{| A|^5}{| A|^{3-\gamma}t^\gamma}
        \gg | A|^{2+\gamma}q^{-\alpha\gamma}
        \gg q^{\big(\frac{d+1}{2}-\frac{\gamma((d+1)-2\alpha)}{2(2+\gamma)}\big)(2+\gamma)}q^{-\alpha\gamma}\\
        &= q^{\frac{1}{2}\big((d+1)(2+\gamma)-\gamma(d+1)\big)+\alpha\gamma}q^{-\alpha\gamma}=q^{d+1}.
    \end{align*}
    Hence, Theorem \ref{twosets} implies $| \Delta_3(A)|=| \Delta_2(A_1,A_2)|\gg q$.
    
    \item[(ii)] Assume $k=4$ and $A$ satisfies
$| A| \gg q^{\frac{d+1}{2}-\frac{\gamma((d+1)-2\alpha)}{2(1+\gamma)}}$.\\
This case is proved in the same way with $A_1=A_2=A+A$.
\begin{flushright}
$\Box$
\end{flushright}
\end{enumerate}

\section{Proofs of Theorems \ref{thm3} and \ref{thm4}}

Let $d\ge 3$, $k\ge 3$ be integers, $j\in \mathbb{F}_q^\ast$. Based on results given by Covert, Koh, Pi \cite{Co} and Hieu, Pham \cite{HP17} as mentioned in the introduction, to prove Theorems \ref{thm3} and \ref{thm4}, it suffices to assume that $|A|\le  q^{\frac{d-1}{2}+\frac{1}{k-1}}$.

We begin by giving an estimate for energies of sets on spheres.
\begin{lemma}\label{energyk}
Let $d\ge 3, l\ge 2$ be integers and $j$ be an nonzero element in $\mathbb{F}_q$. Assume that one of the followings holds:
\begin{itemize}
    \item[(i)] $d$ is even.
    \item[(ii)] $d=4k'-1$ and $q=1\mod 4$ or $d=4k'+1$; $j$ is a primitive element in $\mathbb{F}_q^*$.
\end{itemize}
If $A$ be a subset of $S_j^{d-1}$ such that $|A|> q^{\frac{d-1}{2}}$, then
    \begin{align*}
        E_{l}(A)\ll q^{\frac{(d-1)(2l-3)-1}{2}}|A|^2+q^{(d-1)(l-2)-1}| A|^3 +q^{-1}| A|^{2l-1}.
    \end{align*}
In particular, if $q^{\frac{d-1}{2}}< |A|\le q^{\frac{d}{2}}$, then
\begin{align*}
        E_{l}(A)\ll q^{\frac{(d-1)(2l-3)-1}{2}}|A|^2 +q^{-1}| A|^{2l-1}.
\end{align*}
\end{lemma}

\begin{proof}
Assume $d\ge 3$, $l\ge 2$ and $j\in \mathbb{F}_q^*$ which satisfy hypotheses of Lemma \ref{energyk}. 

Let $A$ be a subset of $S_j^{d-1}$, then from Lemma \ref{energykinduction}, we have
\begin{align*}
    E_{l}(A)\ll q^{d-1} E_{l-1}(A)+q^{-1}|A|^{2l-1}.
\end{align*}
Using this and an inductive argument, one obtains
\begin{align*}
    E_{l}(A)\ll q^{(d-1)(l-2)}E_2(A)+q^{-1}|A|^{2l-1}\sum\limits_{i=0}^{l-3}\bigg(\frac{q^{d-1}}{|A|^2}\bigg)^i.
\end{align*}
In addition, the condition $|A|> q^{\frac{d-1}{2}}$ implies that
\begin{align}\label{eqenergyk2}
    E_{l}(A)\ll q^{(d-1)(l-2)}E_2(A)+q^{-1}|A|^{2l-1}.
\end{align}
The first statement of Lemma \ref{energyk} follows from Theorem \ref{energydeven} and Theorem \ref{energydood}.

Now assume that $|A|\le q^{\frac{d}{2}}$. By Theorem \ref{energydeven} and Theorem \ref{energydood}, we can check that
\begin{align*}
    E_2(A)\ll  q^{-1}| A|^3+q^{\frac{d-2}{2}}| A|^2\ll q^{\frac{d-2}{2}}| A|^2.
\end{align*}
Substituting this into (\ref{eqenergyk2}), we complete the proof of Lemma \ref{energyk}.
\end{proof}

\subsection{Proof of Theorem \ref{thm3}}
Suppose $d\ge 4$ is even. 
\begin{enumerate}
    \item[(i)] For $k>3$, we assume $A$ satisfies
\begin{align}\label{eqthm3.1}
    C_kq^{\frac{d-1}{2}+\frac{1}{4(k-2)}}\le | A|\le q^{\frac{d-1}{2}+\frac{1}{k-1}},\hspace{1cm}\text{where } C_k>0.
\end{align}
It follows from Lemma \ref{cardak}, Theorem \ref{energyk} and (\ref{eqthm3.1}) that there exists $c_1>0$ such that
\begin{align*}
    | A_{k-1}| 
    &\ge \frac{|A|^{2(k-1)}}{E_{k-1}(A)}
    \ge c_1\frac{|A|^{2(k-1)}}{q^{\frac{(d-1)(2(k-1)-3)-1}{2}}|A|^2 +q^{-1}| A|^{2(k-1)-1}}\\
    &\ge \frac{c_1}{2}\frac{|A|^{2k-2}}{\max\big\{q^{\frac{(d-1)(2k-5)-1}{2}}|A|^2 ,q^{-1}| A|^{2k-3}\big\}}\\
    &\ge \frac{c_1}{2}\min\big\{|A|^{2k-4}q^{\frac{-(d-1)(2k-5)+1}{2}},|A|q\big\}\\
    &\ge \frac{c_1}{2}\min\big\{\big(C_kq^{\frac{d-1}{2}+\frac{1}{4(k-2)}}\big)^{2k-4}q^{\frac{-(d-1)(2k-5)+1}{2}},C_kq^{\frac{d-1}{2}+\frac{1}{4(k-2)}}q\big\}\\
    &\ge\frac{c_1\min\{C_k^{2k-4},C_k\}}{2}\min\big\{q^{\frac{d+1}{2}},q^{\frac{d+1}{2}+\frac{1}{4(k-2)}}\big\}
    =c_2q^{\frac{d+1}{2}},
\end{align*}
where we put $c_2=\frac{c_1\min\{C_k^{2k-4},C_k\}}{2}$.

Combining this with (\ref{eqthm3.1}) one gets
\begin{align*}
    |A||A_{k-1}|
    &\ge \frac{c_1\min\{C_k^{2k-3},C_k^2\}}{2}q^{\frac{d-1}{2}+\frac{1}{4(k-2)}}q^{\frac{d+1}{2}}> c_3q^d,
\end{align*}
where $c_3=\frac{c_1\min\{C_k^{2k-3},C_k^2\}}{2}$.

Choosing $C_k>0$ such that $c_2>2$ and $c_3>16$, we have the following lower bounds:
\begin{align*}
    | A_{k-1}| > 2q^{\frac{d+1}{2}},\hspace{0.5cm}| A|| A_{k-1}| >16 q^d.
\end{align*}
Now since $\Delta_k(A) = \Delta_2(A,A_{k-1})$, applying Theorem \ref{devenAB2} for two set $A$ and $B=A_{k-1}$, we obtains
\begin{align*}
    | \Delta_k(A)| =| \Delta_2(A,A_{k-1})| \gg q.
\end{align*}

\item[(ii)]  For $k=3$, we assume
\begin{align}\label{eqthm3.2}
    C_3q^{\frac{d}{2}-\frac{1}{4}}\le | A|\le q^{\frac{d}{2}},\hspace{1cm}\text{where } C_3>0.
\end{align}
In the same way, apply Lemma \ref{cardak}, Theorem \ref{energydeven} and (\ref{eqthm3.2}), there is $c_1'>0$ such that
\begin{align*}
    | A_2|
    &\ge\frac{| A|^4}{E_2(A)}
    \ge c_1'\frac{| A|^4}{\ds\frac{| A|^3}{q}+q^{\frac{d-2}{2}}| A|^2}\
    \ge \frac{c_1'}{2}\frac{| A|^2}{q^{\frac{d-2}{2}}}\\
    &\ge\frac{c_1'}{2}C_3^2 \big(q^{\frac{d}{2}-\frac{1}{4}}\big)^2 q^{-\frac{d-2}{2}}=\frac{c_1'}{2}C_3^2 q^{\frac{d+1}{2}},
\end{align*}
The third inequality holds since $|A|\leq q^{\frac{d}{2}}$.

Moreover, using (\ref{eqthm3.2}) again, we have
\begin{align*}
    |A||A_2|\ge \frac{c_1'}{2}C_2^3q^{\frac{d}{2}-\frac{1}{4}}q^{\frac{d+1}{2}}>\frac{c_1'}{2}C_2^3q^{d}.
\end{align*}
Choose $C_3>0$ such that $\frac{c_1'}{2}C_3^2>2$ and $\frac{c_1'}{2}C_2^3>16$, then from Theorem \ref{devenAB2}, we deduce that
\begin{align*}
    |\Delta_3(A)| =| \Delta_2(A,A_2)| \gg q.
\end{align*}
\begin{flushright}
$\Box$
\end{flushright}
\end{enumerate}

\subsection{Proof of Theorem \ref{thm4}}
Suppose $d\ge 3$ is an odd integer, either $d=4l+1,l\in \N$ or $d=4l-1$ and $q=1\mod 4$, and $j$ is a primitive element in $\mathbb{F}_q^*$. Note that from this assumption, then $j$ is not a square in $\mathcal{F}_q$.

\begin{enumerate}
    \item[(i)] For $k>3$, assume $A$ satisfies
\begin{align}\label{eqthm4.1}
    C_k'q^{\frac{d-1}{2}+\frac{1}{2(2k-3)}}\le | A|\le  q^{\frac{d-1}{2}+\frac{1}{(k-1)}},
\end{align}
where $C_k'>0$ will be chosen later.

By the same method, from Lemma \ref{cardak}, Theorem \ref{energyk} and (\ref{eqthm4.1}) one can check that
\begin{align*}
    | A| | A_{k-1}|
    &\ge\frac{| A|^{2(k-1)+1}}{E_{k-1}(A)}\\
    &\ge c_1\frac{| A|^{2(k-1)+1}}{q^{\frac{(d-1)(2(k-1)-3)-1}{2}}|A|^2 +q^{-1}| A|^{2(k-1)-1}}\\
    &\ge \frac{c_1}{2} \min\big\{|A|^{2k-3}q^{\frac{-(d-1)(2k-5)+1}{2}},|A|^2q\big\}\\
    &\ge \frac{c_1}{2} \min\big\{\big(C_k'q^{\frac{d-1}{2}+\frac{1}{2(2k-3)}}\big)^{2k-3} q^{\frac{-(d-1)(2k-5)+1}{2}}, \big(C_k'q^{\frac{d-1}{2}+\frac{1}{2(2k-3)}}\big)^2q\big\}
    \\
    &\ge \frac{c_1\min\{(C_k')^{2k-3},(C_k')^2\}}{2}q^d.
\end{align*}
Choose $C_k'>0$ such that $| A| | A_{k-1}|\ge 4 q^d$. Then, it follows from Theorem \ref{doodAB} that
\begin{align*}
    | \Delta_k(A)| =| \Delta_2(A,A_{k-1})| \gg q.
\end{align*}

\item[(ii)] For $k=3$, assume $A$ satisfies
\begin{align}\label{eqthm4.2}
    C_3'q^{\frac{d}{2}-\frac{1}{3}}\le | A|\le  q^{\frac{d}{2}}, \text{ where } C_3'>0.
\end{align}
In the same way, applying Lemma \ref{cardak}, Theorem \ref{energydood}, and (\ref{eqthm4.2}), one obtains
\begin{align*}
    | A| | A_2| 
    &\ge| A|\frac{| A|^4}{E_2(A)}\ge \frac{c_1'}{2}C_3'^3 q^{\frac{3d}{2}-1} q^{-\frac{d}{2}+1}
    =\frac{c_1'}{2}C_3'^3 q^{d},
\end{align*}
for some constant $c_1'>0$.

Choose $C_3'>0$ such that $|A|| A_2|\ge 4q^d$.
We complete the proof by applying Theorem \ref{doodAB}.
\begin{flushright}
$\Box$
\end{flushright}
\end{enumerate}

\section*{Acknowledgments}

I would like to thanks Doowon Koh, Thang Pham, Chun-Yen Shen for useful discussions and comments.

\end{document}